\documentclass[12pt]{amsart}
\usepackage[latin1]{inputenc}
\usepackage{color}
\usepackage{enumerate}
\usepackage{amssymb}
\usepackage[all]{xy}
\usepackage{hyperref}
\usepackage{verbatim} 

\newtheorem{thm}{Theorem}[section]

\newtheorem{lem}[thm]{Lemma}
\newtheorem{prop}[thm]{Proposition}
\theoremstyle{definition}
\newtheorem{defn}[thm]{Definition}
\theoremstyle{remark}

\numberwithin{equation}{section}
\newcommand{\norm}[1]{\left\Vert#1\right\Vert}
\newcommand{\abs}[1]{\left\vert#1\right\vert}
\newcommand{\set}[1]{\left\{#1\right\}}
\newcommand{\Real}{\mathbb R}

\newcommand{\To}{\longrightarrow}

\newcommand{\nat}{\mathbb{N}}

\usepackage{color}
\usepackage{enumerate}
\begin{document}

\setcounter{tocdepth}{1}


\title[On the Banach lattice $c_0$]{On the Banach lattice $c_0$}
\author[A.\ Avil\'es]{Antonio Avil\'es}
\address{Universidad de Murcia, Departamento de Matem\'{a}ticas, Campus de Espinardo 30100 Murcia, Spain.}
\email{avileslo@um.es}

\author[G. Mart\'inez-Cervantes]{Gonzalo Mart\'inez-Cervantes}
\address{Universidad de Murcia, Departamento de Matem\'{a}ticas, Campus de Espinardo 30100 Murcia, Spain.}
\email{gonzalo.martinez2@um.es}

\author[J.D. Rodr\'iguez Abell\'an]{Jos\'e David Rodr\'iguez Abell\'an}
\address{Universidad de Murcia, Departamento de Matem\'{a}ticas, Campus de Espinardo 30100 Murcia, Spain.}
\email{josedavid.rodriguez@um.es}

\thanks{Authors supported by project MTM2017-86182-P (Government of Spain, AEI/FEDER, EU) and project 20797/PI/18 by Fundaci\'{o}n S\'{e}neca, ACyT Regi\'{o}n de Murcia. Third author supported by FPI contract of Fundaci\'on S\'eneca, ACyT Regi\'{o}n de Murcia.}



\begin{abstract}

We show that $c_0$ is not a projective Banach lattice, answering a question of B. de Pagter and A. Wickstead. On the other hand, we show that $c_0$ is complemented in the free Banach lattice generated by itself (seen as a Banach space). As a consequence, the free Banach lattice generated by $c_0$ is not projective.

\end{abstract}

\maketitle

\setlength{\parskip}{4mm}

\section{Introduction}

The purpose of this paper is to answer negatively Question 12.11 proposed by B. de Pagter and A. Wickstead in \cite{dPW15} (notice that this also answers negatively \cite[Question 12.10]{dPW15}). We prove that $c_0$, seen as a Banach lattice, is not projective. Moreover, we show that it can be (isometrically) embedded as a Banach lattice into the free Banach lattice generated by itself seen as a Banach space, which is denoted by $FBL[c_0]$. This embedding composed with the natural quotient from $FBL[c_0]$ onto $c_0$ gives the identity map on $c_0$. Thus, $c_0$ is complemented in $FBL[c_0]$. As a consequence, we will obtain that $FBL[c_0]$ is not  projective.

The concepts of free and projective Banach lattices were introduced in \cite{dPW15}. If $A$ is a set with no extra structure, the free Banach lattice generated by $A$, denoted by $FBL(A)$, is a Banach lattice together with a bounded map $u : A \longrightarrow FBL(A)$ having the following universal property: for every Banach lattice $Y$ and every bounded map $v : A \longrightarrow Y$ there is a unique Banach lattice homomorphism $S : FBL(A) \longrightarrow Y$ such that $S \circ u = v$ and $\norm{S} = \sup \set{\norm{v(a)} : a \in A}$. The same idea is applied by A. Avil\'{e}s, J. Rodr\'{i}guez and P. Tradacete to define the concept of the free Banach lattice generated by a Banach space $E$, $FBL[E]$. This is a Banach lattice together with a bounded operator $u:E\To FBL[E]$  such that for every Banach lattice $Y$ and every bounded operator $T : E \longrightarrow Y$ there is a unique Banach lattice homomorphism $S : FBL[E] \longrightarrow Y$ such that $S \circ u = v$ and $\norm{S} = \norm{T}$.

In \cite{ART18} and \cite{dPW15}, the corresponding authors show that both objects exist and are unique up to Banach lattices isometries. Moreover, A. Avil\'{e}s, J. Rodr\'{i}guez and P. Tradacete give an explicit description of them in \cite{ART18}.

Let $A$ be a non-empty set. For $x \in A$, let $\delta_x:[-1,1]^A \longrightarrow [-1,1]$ be the evaluation function given by $\delta_x(x^*) = x^*(x)$ for every $x^* \in [-1,1]^A$, and for every $f:[-1,1]^A \longrightarrow \mathbb{R}$ define $$\|f\| = \sup \set{\sum_{i = 1}^n \abs{ f(x_{i}^{\ast})} :  n \in \mathbb{N}, \, x_1^{\ast}, \ldots, x_n^{\ast} \in [-1,1]^A, \text{ }\sup_{x \in A} \sum_{i=1}^n \abs{x_i^{\ast}(x)} \leq 1 }.$$ 

The Banach lattice $FBL(A)$ is the Banach lattice generated by the evaluation functions $\delta_x$ inside the Banach lattice of all functions $f:[-1,1]^A\To\mathbb{R}$ with finite norm. The natural identification of $A$ inside $FBL(A)$ is given by the map $u: A \longrightarrow FBL(A)$ where $u(x) = \delta_x$. Since every function in $FBL(A)$ is a uniform limit of such functions, they are all continuous (with respect to the product topology) and positively homogeneous, i.e. they commute with multiplication by positive scalars.

Now, let $E$ be a Banach space. For a function $f:E^\ast \To \mathbb{R}$ consider the norm $$ \norm{f}_{FBL[E]} = \sup \set{\sum_{i = 1}^n \abs{f(x_{i}^{*})} : n \in \mathbb{N}, \, x_1^{*}, \ldots, x_n^{*} \in E^{*},\text{ }\sup_{x \in B_E} \sum_{i=1}^n \abs{x_i^{*}(x)} \leq 1 }.$$

The Banach lattice $FBL[E]$ is the closure of the vector lattice in $\mathbb R^{E^*}$ generated by the evaluations $\delta_x: x^\ast \mapsto x^\ast(x)$ with $x\in E$. These evaluations form the natural copy of $E$ inside $FBL[E]$. All the functions in $FBL[E]$ are positively homogeneous and $weak^\ast$-continuous  when restricted to the closed unit ball $B_{E^\ast}$.
	
The notions of free and projective objects are closely related in the general theory of categories. In the setting of Banach lattices, de Pagter and Wickstead \cite{dPW15} introduced projectivity in the following form:

\begin{defn}\label{projdef} A Banach lattice $P$ is \textit{projective} if whenever $X$ is a Banach lattice, $J$ a closed ideal in $X$ and $Q : X \longrightarrow X/J$ the quotient map, then for every Banach lattice homomorphism $T : P \longrightarrow X/J$ and $\varepsilon > 0$, there is a Banach lattice homomorphism $\hat{T} : P \longrightarrow X$ such that $T = Q \circ \hat{T}$ and $\|\hat{T}\| \leq (1 + \varepsilon)\norm{T}$.
\end{defn}

Some examples of projective Banach lattices given in \cite{dPW15} include $FBL(A)$, $\ell_1$, all finite dimensional Banach lattices and Banach lattices of the form $C(K)$, where $K$ is a compact neighborhood retract of $\mathbb{R}^n$. They also prove that $\ell_{\infty}$ and $c$ are not projective. In this paper we will focus on $c_0$ and $FBL[c_0]$.

\section{Non-projectivity of $c_0$ as a Banach lattice}

In this section we are going to prove that $c_0$, seen as a Banach lattice, is not projective. We will use the following fact (see \cite[Proposition 2.1]{ARA19}):

\begin{prop}\label{quotientofprojective}

Let $P$ be a projective Banach lattice, $\mathcal{I}$ an ideal of $P$ and $\pi:P\To P/\mathcal{I}$ the quotient map. The quotient $P/\mathcal{I}$ is projective if and only if for every $\varepsilon>0$ there exists a Banach lattice homomorphism $u_\varepsilon:P/\mathcal{I}\To P$ such that $\pi\circ u_\varepsilon = id_{P/\mathcal{I}}$ and $\|u_\varepsilon\|\leq 1+\varepsilon$.

\end{prop}

Let $L = \mathcal{P}_{fin}^+ (\omega) = \mathcal{P}_{fin}(\omega) \setminus \set{\emptyset}$ be the set of the finite parts of $\omega$ without the empty set.

For $A \in L$ let us define the map $\chi_A : L \longrightarrow [-1,1]$ given by $\chi_A(B) = 1$ if $B \subset A$ and $\chi_A(B) = 0$ if $B \not\subset A$.

Let $\Phi: FBL(L) \longrightarrow c_0$ be the map given by $$\Phi(f) = \big( f \left( \left( \chi_A( \set{1}) \right)_{A \in L} \right), f \left( \left( \chi_A( \set{2}) \right)_{A \in L} \right), \ldots \big) = \big( f \left( \left( \chi_A( \set{n}) \right)_{A \in L} \right) \big)_{n \in \mathbb{N}}$$ for every $f: [-1,1]^L \longrightarrow \mathbb{R} \in FBL(L)$.

\begin{lem}
\label{Lemm2}
The map $\Phi: FBL(L) \longrightarrow c_0$ has the following properties:

\begin{enumerate}

\item $\Phi(\delta_A) = \sum_{i \in A}e_i \in c_0$ for every $A \in L$.

\item $\Phi$ is a Banach lattice homomorphism that is well-defined, i.e., $\Phi(f) \in c_0$ for every $f \in FBL(L)$.

\item $\Phi$ is surjective.

\end{enumerate}

\end{lem}

\begin{proof}

The first assertion follows from the definition of $\Phi$.
For every $n\in \nat$ let $g_n: L \longrightarrow [-1,1]$ be the function  $g_n= \left( \chi_A( \set{n}) \right)_{A \in L}.$ Then the sequence $\left(g_n\right)_{n \in \nat}$ is pointwise convergent to zero, so $f(g_n)$ converges to zero and $\Phi(f)\in c_0$. 
Since $\Phi$ preserves linear combinations, suprema, infima and $\norm{\Phi(f)}\leq \norm{f}$ for every $f \in FBL(L)$, we have that $\Phi$ is a Banach lattice homomorphism whose image is in $c_0$.

Let us prove property \textit{(3)}.
Let $x = (x_1, x_2, \ldots) \in c_0$ and suppose, without loss of generality, that $x \geq 0$.
Fix a sequence of natural numbers $\left(n_i\right)_{i \in \nat}$ such that $x_{n_1} \geq x_n$ for every $n \in \mathbb{N}$ and $x_{n_{i+1}} \geq x_n$ for every $n \in \mathbb{N} \setminus \set{n_1 \ldots, n_i}$.

Now, let $A_i = \set{n_1, \ldots, n_i}$ and $\lambda_i = x_{n_i} - x_{n_{i+1}}$ for every $i = 1, 2, \ldots$

For $A \in L$, if we put $e_A := \sum_{i \in A}e_i$, we have that $x = \sum_{j=1}^{\infty}\lambda_j e_{A_j}$, and then, $$x = \sum_{j=1}^{\infty}\lambda_j \Phi(\delta_{A_j}) = \sum_{j=1}^{\infty}\Phi(\lambda_j \delta_{A_j}) = \Phi \big(\sum_{j=1}^{\infty}\lambda_j \delta_{A_j}\big),$$

where the last element $\sum_{j=1}^{\infty}\lambda_j \delta_{A_j}$ is well-defined since $\sum_{j=1}^{\infty}\lambda_j < \infty$ and each $\delta_{A_j}$ has norm one.
\end{proof}

Thus, $\Phi$ is a quotient map. We are going to prove that there is no bounded Banach lattice homomorphism $\varphi: c_0 \longrightarrow FBL(L)$ such that $\Phi \circ \varphi = id_{c_0}$. This fact will be a consequence of the following Lemma:

\begin{lem}
\label{lemmaaux}
Let $A$ be an infinite set, $(x_n^*)_{ n \in \nat }$ a sequence in $[-1,1]^A$ and $(f_n)_{ n \in \nat }$  a sequence in $FBL(A)$ with the following properties:
\begin{enumerate}
	\item $f_n \geq0 $ for every $n\in \nat$;
	\item $f_n(x_n^*)=1$ for every $n \in \nat$;
	\item For every finite set $F\subseteq A$ there is a natural number $n$ such that $x_n^*|_F=0$, i.e. the restriction of $x_n^*$ to $F$ is null.
\end{enumerate}
Then for every $\varepsilon>0$ there is a subsequence $(f_{n_k})_{ k \in \nat }$ such that $$\norm{ \sum_{k=1}^{m} f_{n_k} } \geq m - \varepsilon \mbox{ for every }m\in \nat.$$
\end{lem}
\begin{proof}
Fix $\varepsilon>0$ and $f_{n_1}:=f_1$. Since the elements of $FBL(A)$ are continuous with respect to the product topology, there is a neighborhood $U_1$ of $x_1^*$ such that $f_1(x^*)\geq 1-\frac{\varepsilon}{2}$ whenever $x^* \in U_1$.
In particular, there is a finite set $F_1 \subseteq A$ such that $f_1(x^*)\geq 1-\frac{\varepsilon}{2}$ whenever $x^*|_{F_1} = x_1^*|_{F_1}$.

We recursively construct the subsequence $(f_{n_k})_{ k \in \nat }$ and the sequence of sets $(F_k)_{ k \in \nat }$. Suppose that we have $f_{n_1}, \ldots, f_{n_k}$ and $F_1 , \ldots, F_{k}$ finite subsets of $A$ such that $x_{n_i}^*|_{F_1 \cup F_2 \cup \ldots \cup F_{i-1}}=0$ and
$f_{n_i}(x^*)\geq 1-\frac{\varepsilon}{2^i}$ whenever $x^*|_{F_i} = x_{n_i}^*|_{F_i}$.

Property $(3)$ guarantees the existence of a number $n_{k+1} \in \nat$ such that $x_{n_{k+1}}^*|_{F_1 \cup F_2 \cup \ldots \cup F_{k}}=0$.
It follows from property $(2)$ that there is a finite set $F_{k+1} \subseteq A$ such that $f_{n_{k+1}}(x^*)\geq 1-\frac{\varepsilon}{2^{k+1}}$ whenever $x^*|_{F_{k+1}} = x_{n_{k+1}}^*|_{F_{k+1}}$.

For each $k\in \nat$ define $y^*_k: A \longrightarrow [-1,1]$ such that $y_k^*|_{F_k}=x_{n_k}^*|_{F_k}$ and $y_k^*(x)=0$ whenever $x \in A \setminus F_k$.
Notice that $f_{n_k}(y_k^*) \geq 1-\frac{\varepsilon}{2^k}$ for every $k\in \nat$.
On the other hand, if $k<k'$ and $y_{k}^*(x) \neq 0$ then $x \in F_k$ (by the definition of $y_k^*$) and therefore $x^*_{n_{k'}}(x)=0$, so $y^*_{k'}(x)=0$. It follows that $y^*_k$ and $y^*_{k'}$ have disjoint supports. In particular,
$$ \sup_{x \in A} \sum_{k=1}^m \abs{y_k^{\ast}(x)} \leq 1. $$

Thus,

$$\norm{ \sum_{k=1}^{m} f_{n_k} } = 
\sup \set{\sum_{i = 1}^n \abs{\sum_{k=1}^{m} f_{n_k}(z_{i}^{*})} : n \in \mathbb{N}, \, z_1^{*}, \ldots, z_n^{*} \in [-1,1]^A,\text{ }\sup_{x \in A} \sum_{i=1}^n \abs{z_i^{*}(x)} \leq 1 } \geq $$
$$ \geq \sum_{i = 1}^m \abs{\sum_{k=1}^{m} f_{n_k}(y_{i}^{*})}
\overset{\mbox{(1)}}{\geq} \sum_{k=1}^{m} f_{n_k}(y_{k}^{*}) \geq \sum_{k=1}^{m} \left( 1-\frac{\varepsilon}{2^k} \right)\geq m - \varepsilon $$

for every $m\in \nat$.
\end{proof}

\begin{thm}
The Banach lattice $c_0$ is not projective.
\end{thm}
\begin{proof}
We argue by contradiction. Suppose $c_0$ is projective. Since $FBL(L)$ is projective, it follows from Proposition \ref{quotientofprojective} and Lemma \ref{Lemm2} the existence of a  bounded Banach lattice homomorphism $\varphi: c_0 \longrightarrow FBL(L)$ such that $\Phi \circ \varphi = id_{c_0}$. Set $f_n:= \varphi (e_n)$. Since $\varphi$ is a Banach lattice homomorphism and each $e_n$ is positive, we have that $f_n=\varphi(e_n)\geq 0$ for every $n\in \nat$. It follows from the equality $\Phi(f_n)=(\Phi \circ \varphi)(e_n) = e_n$  and the definition of $\Phi$ that $$f_n \left( \left( \chi_A( \set{n}) \right)_{A \in L} \right) = e_n(n)= 1$$
for every $n \in \nat$. Set $x_n^*=\left(\chi_A ( \set{n}) \right)_{A\in L}$ for every $n\in \nat$.
Notice that for every finite set $F\subseteq L$ we have  $x_n^* (S)=0$ for every $S\in F$ whenever $n \notin \bigcup_{S\in F} S$.

Thus, Lemma \ref{lemmaaux} asserts that 
for every $\varepsilon>0$ there is a subsequence $(f_{n_k})_{ k \in \nat }$ such that $$\norm{ \sum_{k=1}^{m} f_{n_k} } \geq m - \varepsilon \mbox{ for every }m\in \nat.$$
On the other hand, since $\varphi$ is bounded, there is a constant $C>0$ such that 
$$ \norm{ \sum_{k=1}^{m} f_{n_k} } = \norm{ \varphi \left( \sum_{k=1}^{m} e_{n_k} \right)}\leq C\norm{\sum_{k=1}^{m} e_{n_k}}_{\infty}=C \mbox{ for every }m\in \nat.$$

Thus, $m -\varepsilon \leq C$ for every $\varepsilon >0$ and every $m\in \nat$, which yields to a contradiction.

\end{proof}

\section{Complementability of $c_0$ in $FBL[c_0]$}

This section is devoted to the proof that $c_0$ is \textit{lattice-embeddable} in $FBL[c_0]$ as a Banach lattice, that is to say, there exist a Banach lattice homomorphism $u: c_0 \longrightarrow FBL[c_0]$ and two constants $K, M \geq 0$ such that $$K\norm{x}_{\infty} \leq \norm{u(x)}_{FBL[c_0]} \leq M\norm{x}_{\infty}$$ for every $x \in c_0.$ Moreover, we will prove that $c_0$ is complemented in $FBL[c_0]$.

By \cite[Theorem 4.50]{CB} we know that the Banach lattice $c_0$ is lattice-embeddable in a Banach lattice $E$ if and only if there exists a disjoint sequence $(x_n)_{n \in \mathbb{N}} \in E^+$ (the positive cone of $E$) such that 

\begin{itemize}

\item[a)]$(x_n)_{n \in \mathbb{N}}$ does not converge in norm to zero, and

\item[b)]the sequence of partial sums of $(x_n)_{n \in \mathbb{N}}$ is norm bounded, i.e., there exists some $M > 0$ satisfying $\norm{\sum_{i=1}^n x_i}_{E} \leq M$ for every $n \in \mathbb{N}$.

\end{itemize}

Thus, what we are going to do is to construct a sequence $(f_n)_{n \in \mathbb{N}} \in FBL[c_0]$ with the desired properties. The following lemma will be very useful:

%
%
%
%
%
%

\begin{lem}\label{funcionesContinuasPosHomogeneas}

Let $A$ be a set and $f:[-1,1]^A \longrightarrow \mathbb{R}$ a continuous and positively homogeneous function that depends on a finite amount of coordinates, i.e., there exists a finite subset $A_0 \subset A$ and $\tilde{f}:[-1,1]^{A_0} \longrightarrow \mathbb{R}$ such that $f(x^*) = \tilde{f}(x^*\vert_{A_0})$. Then, $f$ is in $FBL(A)$.

\end{lem}

\begin{proof}

The function $\tilde{f}:[-1,1]^{A_0} \longrightarrow \mathbb{R}$ is continuous and positively homogeneous. By \cite[Proposition 5.3]{dPW15}, $\tilde{f}$ is in $FBL(A_0)$.

Let $T: FBL(A_0) \longrightarrow FBL(A)$ be the map given by $T(g)(x^*) = g(x^*\vert_{A_0})$ for every $g:[-1,1]^{A_0} \longrightarrow \mathbb{R}$ and $x^* \in [-1,1]^A$.

Clearly, $f = T(\tilde{f})$, and then $f$ is in $FBL(A)$.

\end{proof}

Let $(N_n)_{n \in \mathbb{N}}$ be a strictly increasing sequence of natural numbers.

For $r \in \mathbb{R}$ let $r^+ = \max\set{r,0}$ be the positive part of $r$, and for every $n \in \mathbb{N}$ let $f_n: c_0^* \longrightarrow \mathbb{R}$ be the map given by $$f_n(x^*) = (|x_n^*|-N_n\max\set{|x_m^*| : m<n})^+ \cdot \Pi_{m > n}g_{nm}(x^*)$$ for every $x^*=(x_1^*,x_2^*,\ldots) \in c_0^*=\ell_1$, where $g_{nm}: c_0^* \longrightarrow [0,1]$ is any continuous function such that $g_{nm}(x^*) = 0$ if $N_m|x_n^*| \leq |x_m^*|$, $g_{nm}(x^*) = 1$ if $|x_m^*| \leq (N_m-1)|x_n^*|$ and $g_{nm}(x^*)=g_{nm}(\frac{x^*}{\norm{x^*}})$ whenever $x^* \neq 0$.

Let us see that $(f_n)_{n \in \mathbb{N}}$ is a disjoint sequence of positive elements which satisfies both properties a) and b):

\begin{lem}

$f_n \geq 0$ for every $n \in \mathbb{N}$ and $f_n \wedge f_l = 0$ for every $n \neq l$.

\end{lem}

\begin{proof}

The first assertion is clear. For the second one, suppose, for example, that $n < l$, and let $x^* \in c_0^*$ such that $f_l(x^*) \neq 0$. We have that $|x_l^*| > N_l \max \set{|x_m^*| : m<l}$. In particular, $|x_l^*| > N_l |x_n^*|$. Now, if $f_n(x^*) \neq 0$, we have that $g_{nm}(x^*) \neq 0$ for every $m > n$, and then, that $|x_m^*| < N_m |x_n^*|$ for every $m > n$. Taking $m = l$ we have a contradiction.

\end{proof}

\begin{lem}
\label{LemAux1}

$f_n$ is in $FBL[c_0]$ for every $n \in \mathbb{N}$.

\end{lem}

\begin{proof}

Fix $n \in \mathbb{N}$. We are going to find a sequence of functions $(h_k)_{k \in \mathbb{N}} \in FBL[c_0]$ such that $\lim_{k \to +\infty} \norm{h_{k} - f_n}_{FBL[c_0]} = 0$. Then, we will have that $f_n \in FBL[c_0]$.

Let $h_{k}: c_0^* \longrightarrow \mathbb{R}$ be the map given by $$h_{k}(x^*) = (|x_n^*|-N_n\max\set{|x_m^*| : m<n})^+ \cdot \Pi_{n < m \leq n+k}~g_{nm}(x^*)$$ for every $x^*=(x_1^*,x_2^*,\ldots) \in c_0^*=\ell_1$. Let us see that $h_k \in FBL[c_0]$ for every $k \in \mathbb{N}$.

Notice that $h_{k}$ is continuous, positively homogeneous and satisfies that $h_k(x^*) = h_{k}(y^*)$ whenever $x_1^* = y_1^*, \ldots, x_{n+k}^* = y_{n+k}^*$. Let $P: FBL(B_{c_0}) \longrightarrow \mathbb{R}^{c_0^*}$ be the map given by $P(f)(x^*) = f((\frac{x^*(y)}{\norm{x^*}})_{y \in B_{c_0}}) \cdot \norm{x^*}$ if $x^* \neq 0$ and $P(f)(0) = f(0)$, where $0$ denotes the identically zero function in the corresponding space. We have that $P(FBL(B_{c_0}
)) \subset FBL[c_0]$ because $P$ maps the evaluation functions in $FBL(B_{c_0})$ to the evaluations functions in $FBL[c_0]$, preserves linear combinations, the lattice structure and $\norm{P(f)}_{FBL[c_0]} \leq \norm{f}_{FBL(B_{c_0})}$.

Now, let $\tilde{h_{k}}: [-1,1]^{B_{c_0}} \longrightarrow \mathbb{R}$ be the map given by $$\tilde{h_{k}}(z^*)=h_k((z^*(e_1),\ldots, z^*(e_{n+k}),0,0,\ldots))$$ for every $z^* \in [-1,1]^{B_{c_0}}$.
Since $\tilde{h_{k}}$ is continuous, positively homogeneous and depends only on finitely many coordinates, by Lemma \ref{funcionesContinuasPosHomogeneas} we have that $\tilde{h_{k}} \in FBL(B_{c_0})$. It follows that $P(\tilde{h_{k}}) = h_{k} \in FBL[c_0]$.

Now, by definition, we have that $$\norm{h_k-f_n}_{FBL[c_0]} = \sup \set{\sum_{i=1}^l |(h_k-f_n)|(x_i^*) : l \in \mathbb{N}, x_1^*, \ldots, x_l^* \in B_{\ell_1}, \sup_{x \in B_{c_0}}\sum_{i=1}^l|x_i^*(x)|\leq 1}.$$

Take $x_1^*, \ldots, x_l^* \in B_{\ell_1}$, with $x_i^* = (x_{i1}^*, x_{i2}^*, \ldots)$ for every $i = 1, \ldots, l$ and such that $(h_k-f_n)(x_i^*) \neq 0$ for every $ i = 1,\ldots, l$ and $\sup_{x \in B_{c_0}}\sum_{i = 1}^l|x_i^*(x)| \leq 1$.

Note that $h_k \geq f_n$, so we can remove absolute values in the previous expression. Then, we have that

\begin{eqnarray*} 
 \sum_{i=1}^l (h_k-f_n)(x_i^*) & = & \sum_{i=1}^l (|x_{in}^{*}| - N_n\max\set{|x_{im}^{*}| : m<n})^+ \cdot \\
 & \cdot & \big[\Pi_{n<m\leq n+k}~g_{nm}(x_i^*) - \Pi_{n<m}~g_{nm}(x_i^*)\big] \\
 & = & \sum_{i=1}^l (|x_{in}^{*}| - N_n\max\set{|x_{im}^{*}| : m<n})^+ \cdot \\
 & \cdot & \big[\Pi_{n<m\leq n+k}~g_{nm}(x_i^*) - \Pi_{n<m \leq n+k}~g_{nm}(x_i^*) \Pi_{n+k < m}~g_{nm}(x_i^*) \big] \\
 & = & \sum_{i=1}^l (|x_{in}^{*}| - N_n\max\set{|x_{im}^{*}| : m<n})^+ \cdot \\
 & \cdot & \Pi_{n<m\leq n+k}~g_{nm}(x_i^*)\big[ 1 - \Pi_{n+k <m}~g_{nm}(x_i^*)\big] \\
 & \leq & \sum_{i=1}^l |x_{in}^*| \big[ 1 - \Pi_{n+k <m}~g_{nm}(x_i^*)\big].
\end{eqnarray*}

Since $(h_{k}-f_n)(x_i^*) \neq 0$ for every $ i = 1,\ldots, l$, we have that $|x_{in}^*| \neq 0$ and also that $1 - \Pi_{n+k <m}~g_{nm}(x_i^*) \neq 0$. Thus, for every $i$ there exists $m_i > n+k$ such that $g_{nm_i}(x_i^*) \neq 1$, that is to say, $\frac{|x_{im_i}^*|}{|x_{in}^*|} > N_{m_i} - 1$. Since $N_{m_i} > N_{n+k}$, this implies that $|x_{in}^*| < \frac{1}{N_{n+k}-1}|x_{im_i}^*|$.

Thus, $$\sum_{i=1}^l (h_{k}-f_n)(x_i^*) \leq \sum_{i=1}^l |x_{in}^*| < \frac{1}{N_{n+k}-1}\sum_{i=1}^l |x_{im_i}^*|.$$

Therefore, since $\lim_{k \rightarrow +\infty}\frac{1}{N_{n+k}-1} =0$, the proof will follow from the following Claim.

\textbf{Claim.}  For every $x_1^*, x_2^*, \ldots, x_l^*\in B_{\ell_1}$ and every natural numbers $m_1, m_2, \ldots, m_l \in \nat$ we have
$$\sum_{i=1}^l |x_{im_i}^*| \leq 1 $$
whenever $\sup_{x \in B_{c_0}}\sum_{i = 1}^l|x_i^*(x)| \leq 1$.

\textit{Proof of the Claim.} Fix $m=\max\{m_i: 1\leq i \leq l \}.$ We show first that $$\sum_{i=1}^l |x_{im_i}^*| \leq \max \set{\sum_{i=1}^l\big|\sum_{j=1}^m \varepsilon (j)x_{ij}^*\big| : \varepsilon \in \set{-1,+1}^m}.$$

In fact,

\begin{eqnarray*} 
 \max \set{\sum_{i=1}^l\big|\sum_{j=1}^m\varepsilon (j)x_{ij}^*\big| : \varepsilon \in \set{-1,+1}^m} & \geq & \frac{1}{2^m}\sum_{\varepsilon \in \set{-1,+1}^m}\sum_{i=1}^l \big| \sum_{j=1}^m \varepsilon(j) x_{ij}^* \big| \\
 & = & \frac{1}{2^m}\sum_{\tilde{\varepsilon} \in \set{-1,+1}}\sum_{i=1}^l \sum_{\substack{\varepsilon \in \set{-1,+1}^m \\ \varepsilon(m_i) = \tilde{\varepsilon}}} \big| \sum_{j=1}^m \varepsilon(j) x_{ij}^* \big| \\
 & \geq & \frac{1}{2}\sum_{\tilde{\varepsilon} \in \set{-1,+1}}\sum_{i=1}^l \frac{1}{2^{m-1}} \big| \sum_{\substack{\varepsilon \in \set{-1,+1}^m \\ \varepsilon(m_i) = \tilde{\varepsilon}}} \sum_{j=1}^m \varepsilon(j) x_{ij}^* \big| \\
 & = &  \sum_{i=1}^l |x_{im_i}^*|.
\end{eqnarray*}

Let $\varepsilon \in \set{-1,+1}^m$ be the function which gives the maximum above. We have that $$\sum_{i=1}^l |x_{im_i}^*| \leq \sum_{i=1}^l\big|\sum_{j=1}^m\varepsilon (j)x_{ij}^*\big| = \sum_{i=1}^l |x_i^*(x)| \leq 1,$$ taking, in the equality, $x = (\varepsilon(1), \varepsilon(2), \ldots, \varepsilon(m), 0, \ldots) \in B_{c_0}$.

\end{proof}

\begin{lem}\label{sumasParcialesfi}

$\|\sum_{i=1}^n f_i \|_{FBL[c_0]} \leq 1$ for every $n \in \mathbb{N}$.

\end{lem}

\begin{proof}

By definition, we have that $$\norm{\sum_{i=1}^n f_i}_{FBL[c_0]} = \sup \set{\sum_{j=1}^k(\sum_{i=1}^nf_i)(x_j^*) : k \in \mathbb{N}, x_1^*, \ldots, x_k^* \in B_{\ell_1}, \sup_{x \in B_{c_0}}\sum_{j=1}^k|x_j^*(x)|\leq 1}.$$

Fix $x_1^*, x_2^*, \ldots, x_k^* \in  B_{\ell_1}$ with $\sup_{x \in B_{c_0}} \sum_{j=1}^k |x_j^*(x)| \leq 1 $.
Since the functions $f_i$ are disjoint, for each $j=1,2,\ldots,k$ there is at most one $i_j \in \{1,2,\ldots,n\}$ such that $f_{i_j}(x_j^*) \neq 0$. Thus,
$$ \sum_{j=1}^k(\sum_{i=1}^nf_i)(x_j^*) = \sum_{j=1}^k f_{i_j}(x_j^*).$$

Without loss of generality, we suppose that $f_{i_j}(x_j^*)\neq 0$ for every $j=1,2, \ldots, k$.  

Notice that each $f_{i_j}(x_j^*) \leq | x_{ji_j}^*| $ for every $j=1,2, \ldots, k$, so 
$$ \sum_{j=1}^k(\sum_{i=1}^nf_i)(x_j^*) = \sum_{j=1}^k f_{i_j}(x_j^*) \leq \sum_{j=1}^k  | x_{ji_j}^*|.$$

Now the Claim in Lemma \ref{LemAux1} applied to $x_1^*, \ldots, x_k^*$ asserts that
$\sum_{j=1}^k  | x_{ji_j}^*| \leq 1$ and therefore
$$ \sum_{j=1}^k(\sum_{i=1}^nf_i)(x_j^*) \leq \sum_{j=1}^k  | x_{ji_j}^*| \leq 1.$$
Thus,  $\norm{\sum_{i=1}^n f_i}_{FBL[c_0]} \leq 1$.

\end{proof}

\begin{lem}

The sequence $(f_n)_{n \in \mathbb{N}}$ does not converge in norm to zero. Furthermore, $\norm{f_n}_{FBL[c_0]} = 1$ for every $n \in \mathbb{N}$.

\end{lem}

\begin{proof}

We know that $$\norm{f_n}_{FBL[c_0]} = \sup \set{\sum_{j=1}^k f_n(x_j^*) : k \in \mathbb{N}, x_1^*, \ldots, x_k^* \in B_{\ell_1}, \sup_{x \in B_{c_0}}\sum_{j=1}^k|x_j^*(x)|\leq 1}.$$

Taking $e_n^* \in B_{\ell_1}$ we have that $f_n(e_n^*) = 1$, so that $\norm{f_n}_{FBL[c_0]} \geq 1$.
In general, since $f_n(x^*)\leq |x^*(e_n)|$ for every $x^* \in B_{\ell_1}$, we have that

$$\sum_{j=1}^k f_n(x_j^*) \leq \sum_{j=1}^k |x_j^*(e_n)| \leq \sup_{x \in B_{c_0}}\sum_{j=1}^k|x_j^*(x)| \leq 1 ,$$
so  $\norm{f_n}_{FBL[c_0]} = 1$.
\end{proof}

Thus, the Banach lattice $c_0$ is lattice-embeddable into the free Banach lattice $FBL[c_0]$ \cite[Theorem 4.50]{CB}. For the sake of completeness we include a proof that, indeed, this embedding has norm one:

\begin{thm}
The operator $u: c_0 \longrightarrow FBL[c_0]$ given by $u(x) = \sum_{i=1}^{+\infty}x_i f_i$ for every $x = (x_1, x_2, \ldots) \in c_0$ is a Banach lattice embedding with $\|u\|=1$.
\end{thm}

\begin{proof}
It is enough to prove that 
$$ \norm{\sum_{i=1}^n a_i f_i }_{FBL[c_0]} = \max_{1 \leq i \leq n} |a_i| $$
for every $n\in \nat$ and every $a_1,a_2,\ldots,a_n \in \Real$.
Fix any $a_1,a_2,\ldots,a_n \in \Real$. Without loss of generality, suppose that $|a_1|=\max_{1 \leq i \leq n} |a_i|$. Since each function $f_i$ is positive and has norm one, we have that

\begin{eqnarray*} 
 \max_{1 \leq i \leq n} |a_i| = |a_1|=\norm{|a_1|f_1}_{FBL[c_0]} \leq  \norm{\sum_{i=1}^n |a_i| f_i}_{FBL[c_0]} & = & \norm{\bigg|\sum_{i=1}^n a_i f_i\bigg|}_{FBL[c_0]}\\
 & = & \norm{\sum_{i=1}^n a_i f_i}_{FBL[c_0]}.
\end{eqnarray*}

On the other hand, 

\begin{eqnarray*} 
\norm{\sum_{i=1}^n a_i f_i}_{FBL[c_0]} =  \norm{\sum_{i=1}^n |a_i| f_i}_{FBL[c_0]} & \leq & \norm{\sum_{i=1}^n |a_1| f_i }_{FBL[c_0]}\\
& = & |a_1| \norm{ \sum_{i=1}^n f_i }_{FBL[c_0]}\\
& \leq & |a_1| = \max_{1 \leq i \leq n} |a_i|.
\end{eqnarray*}

\end{proof}	

The Banach lattice homomorphism $T : FBL[c_0] \longrightarrow c_0$ given by the formula $T(f) = (f(e_1^*), f(e_2^*), \ldots)$ for every $f: c_0^* \longrightarrow \mathbb{R} \in FBL[c_0]$ is surjective (notice that $FLB[c_0]$ contains a natural copy of $c_0$ and $T|_{c_0}$ is the identity map), so it is a quotient map. Since $T \circ u = id_{c_0}$, we have proved the following theorem:

\begin{thm}

The Banach lattice $c_0$ is complemented in $FBL[c_0]$.

\end{thm}

As a consequence of this fact, that $\norm{u}=1$ and that $c_0$ is not projective, applying again Proposition \ref{quotientofprojective}, we conclude that:

\begin{thm}

$FBL[c_0]$ is not projective.

\end{thm}


\begin{thebibliography}{999}

\bibitem{ARA19} A. Avil\'es, J. D. Rodr\'iguez Abell\'an, \textit{Projectivity of the free Banach lattice generated by a lattice}, Archiv der Mathematik. To appear.

\bibitem{ART18} A.\ Avil\'es, J.\  Rodr\'iguez, P.\  Tradacete,  \textit{The free Banach lattice generated by a
Banach space}, J. Funct. Anal. 274 (2018), 2955--2977.

\bibitem{dPW15} B.\ de Pagter,  A. W.\  Wickstead, \textit{Free and projective Banach lattices}, Proc. Royal Soc. Edinburgh Sect. A, 145 (2015), 105--143.

\bibitem{CB} C. D. Aliprantis, O. Burkinshaw, \emph{Positive Operators}, Handbook of the Geometry of Banach Spaces, Springer 2006.

\end{thebibliography}
\end{document}